\documentclass[11pt,psamsfonts,leqno,oneside,letterpaper,left=1.25truein, top=1truein, margin=2.5cm ]{article}
\usepackage[letterpaper,left=1.75truein, top=1truein]{geometry}
\usepackage{amssymb,amsmath, amsthm, amssymb, amscd}
\usepackage{graphicx}

\usepackage{epsfig}
\usepackage{color}

\usepackage[colorlinks,linkcolor=blue,citecolor=blue]{hyperref}
\usepackage{indentfirst}
\usepackage{geometry}

\newtheorem{theorem}{Theorem}[section]
\newtheorem{lemma}[theorem]{Lemma}

\newtheorem{definition}[theorem]{Definition}

\def\Niel{\mbox{\rm{Niel}}}
\def\Bump{\mbox{\rm{Bump}}}
\def\PSL{\mbox{\rm{PSL}}}

\def\Stab{\mbox{\rm{Stab}}}

\title{The bumping set and the characteristic submanifold}
\author{Genevieve S Walsh} 
\date{}
\begin{document}
\maketitle
\begin{abstract} We show here that the Nielsen core of the bumping set of the domain of discontinuity of a Kleinian group $\Gamma$  is the boundary of the characteristic submanifold of the associated 3-manifold with boundary. Some examples of interesting characteristic submanifolds are given. We also give a construction of the characteristic submanifold directly from the Nielsen core of the bumping set. The proofs are from ``first principles", using properties of uniform domains and the fact that quasi-conformal discs are uniform domains. 
 \end{abstract}

\section{Notation and background}

Let $G$ be a Kleinian group without torsion where $\Lambda(G)$ is the limit set and $\Omega(G)$ is the domain of discontinuity.  We denote the quotient 3-manifold with boundary, $({\mathbb H}^3 \cup \Omega(G))/G$, by $M(G)$.  We require throughout that $M(G)$ is geometrically-finite with incompressible, quasi-Fuchsian boundary components.  The group $G$ is finitely-generated and the components of the domain of discontinuity are all disks.  The conformal boundary of $M(G)$, $\Omega(G)/G$, is a finite union of finite area surfaces by Ahlfor's finiteness theorem.  Since each surface subgroup is quasi-Fuchsian, each component of $\Omega(G)$ is a quasi-disc, the image of the standard unit disc in the complex plane under a quasi-conformal homeomorphism.  We will further require that the closure of any one component is a closed disk. In this case we say, by abuse of notation,  that $G$ is a \emph{geometrically-finite Kleinian group with incompressible boundary}. We will be interested in where components of the domain of discontinuity meet.  Accordingly,  define $\Bump(C_1,C_2,...C_n)$ to be $\bar C_1 \cap \bar C_2 \cap...\cap \bar C_n$ where $C_1,...,C_n$  are components of the domain of discontinuity.   Let $C$ be a component of $\Omega(G)$.  The \emph{Nielsen core} of a subset $X$ of $\partial C$ is the convex hull of $X$ in the Poincare metric on $C$.  This is well defined as the Riemann map from the interior of the unit disc to $C$ extends to the boundary by \cite{boundary}.  Let $\Niel_{C_1}(\mathcal{C})$ denote the Nielsen core of the bumping set of components $\mathcal{C} = \lbrace C_1,... C_n \rbrace $ in $C_1$. The \emph{Nielsen core of the bumping set of $\mathcal{C} = C_1,...C_n$} is $\bigcup_i \Niel_{C_i}(\mathcal{C})$, which we denote simply by $\Niel(\mathcal{C})$.   We denote the stabilizer of a component $C$ of $\Omega(G)$ by $\Stab(C)$.  If the Nielsen core of any bumping set is non-trivial, this is an obstruction to $M(G)$ admitting a hyperbolic metric with totally geodesic boundary, and by work of Thurston (see \cite[Theorem 6.2.1]{Mardenbook}) this is the only obstruction.  Maskit showed: 

\begin{lemma} \cite[Theorem 3]{component} $\Lambda(G_C \cap G_B) = \Lambda(G_C) \cap \Lambda(G_B) = \bar  C \cap \bar B$.
\end{lemma} 

This was later generalized by Anderson \cite{Anderson}. 

\section{The image of the Nielsen core of the bumping set} 
Here we show that the image of the Nielsen core of the bumping set of two or more components is a union of  simple closed essential curves  and essential subsurfaces of the boundary components.  This image, with any simple closed curves thickened, will form the boundary of the characteristic submanifold of the quotient $M(G)$. Recall that the stabilizer $G_C \subset G$ of a component $C$ of $\Omega(G)$ has a representation $\phi$ into $\PSL(2, \mathbb{R})$ induced by the uniformization map from the unit disc.   An \emph{accidental parabolic} is an element $g \in G_C$ that is parabolic in $G$ but where $\phi(g)$ is hyperbolic.   Note that our definition of Kleinian group with incompressible boundary, requiring that the closure of any component is a disc, rules out accidental parabolics.  


\begin{theorem} \label{thm:boundary} Let $G$ be a geometrically-finite Kleinian group with incompressible boundary and let $p: {\bf H}^3 \cup \Omega(G) \rightarrow M(G)$ be the covering map induced by the action of $G$. Suppose that $C_1, ..., C_n$  are components of $\Omega(G)$ with non-trivial Nielsen core of the bumping set $\Bump(C_1,..., C_n)$.  Then the image of $\Niel_{C_1}(C_1,....,C_n)$ in $\partial M(G)$ is either a simple geodesic or a subsurface of $p(C_1)$ bounded by geodesics.  \end{theorem}

We note that a similar theorem is proven in Maskit \cite{component} although this uses the deep work of the combination theorems which we do not use.  There is also a similar statement in \cite{Cyril}.  Bill Thurston understood the characteristic submanifold from a similar point of view in his discussion of the window in \cite{BillWindow}.  However, it has come to our attention that proving this statement directly from the properties of quasi-discs may be useful.

\begin{proof} 

We will consider the image of a boundary curve $\beta$ of $\Niel_{C_1}(C_1,..., C_n)$.  The curve $\beta$ is necessarily a geodesic since it is the boundary curve of a convex hull. If the convex hull consists of a geodesic going between the two points of the bumping set, we consider this to be a boundary curve.  The strategy of the proof is the following. We show that the image of $\beta$ (1) is simple, (2) does not accumulate, and (3) does not exit a cusp.  Therefore the image of a boundary curve $\beta$ is an essential simple closed curve on the surface $p(C_1)$.   Since the map is a covering map, the image will be bounded by essential simple closed curves, and hence will be either a simple closed curve or a subsurface of $p(C_1)$.   

(1) The image of a boundary curve $\beta$ is simple. The pre-image of the curve $p(\beta)$ in $C_1$ is the orbit of $\beta$ under the action of $G_{C_1}$.  If the image were not simple, then its pre-images in $C_1$ would intersect. Thus assume that there is a $\gamma$ in the stabilizer of $C_1$ such that $\gamma(\beta)$ intersects  $\beta$ transversely.  Since $\beta$ is the boundary curve of a convex set, this implies $\gamma \notin \Stab(C_i)$ for some $i \in \lbrace 2, 3, ..., n \rbrace$.  Two circles on the two-sphere must intersect an even number of times. Consider the convex hulls of the endpoints of $\beta$ in $\bar C_1$ and $\bar C_i$ and the endpoints of $\gamma(\beta)$ in $\bar C_1$ and $\gamma(\bar C_i)$.    These are two circles in $\bar C_1 \cup \bar C_i \cup \gamma(\bar C_i)$ which intersect once in $C_1$ and which do not intersect on the boundaries of the components.  Thus $C_i$ and $\gamma(C_i)$  intersect in their interior, which contradicts the fact that they are distinct components.  

(2) The image of a boundary curve does not accumulate.  Here we use that a quasi-disc is a uniform domain. 
\begin{definition} \label{d:uniform} A domain $A$ is {\em uniform} if there are constants $a$ and $b$ such that every pair of points $z_1, z_2 \in C$ can be joined by an arc $\alpha$ in $A$ with the following properties: 
\begin{enumerate} 
\item The Euclidean length of $\alpha$ satisfies $$l(\alpha) \leq a|z_1-z_2|$$

\item  For every $z \in \alpha$, 
$$min(l(\alpha_1), l(\alpha_2)) \leq bd(z, \partial A),$$

where $\alpha_1$ and $\alpha_2$ are the components of $\alpha \setminus z$. 
\end{enumerate}
\end{definition} 

By \cite[Part I, Thm 6.2]{Lehto}, a $K$-quasi-disc is a uniform domain with constants $a$ and $b$ which depend only on $K$.  Now suppose that the image of a boundary curve $\beta$ of $\Niel_{C_1}(C_1,..,C_n)$  accumulates in $p(C_1)$.  Then the images of $\beta$  in $C_1$ under the action of $\Stab(C_1)$ accumulate. Since $\beta$ is geodesic, there is a sequence $\lbrace \gamma_i \rbrace$ in $\Stab(C_1)$ such that the endpoints of $\gamma_i(\beta)$ accumulate in $\partial(C_1)$ in the Poincar\'e metric to two points $p$ and $q$.   Since $\beta$ is a boundary curve, there is some component $B \in \lbrace C_2, ... , C_n \rbrace$ such that for an infinite number of $\gamma_i$, $\gamma_i(B)$ are all distinct from each other and from $B$.  We continue to call this subsequence $\lbrace \gamma_i \rbrace$. Since the $\gamma_i$ all act conformally on $S^2_\infty$, each $\partial \gamma_i(B)$ is a $K_B$-quasicircle where $K_B$ is the quasi-conformal constant of $B$.  The point is that the $\gamma_i(B)$ will eventually be too skinny to satisfy a fixed $b$ in condition 2. above. 

There are points $p_i$ and $q_i$ of the $\gamma_i(B)$ which are accumulating to $p$ and $q$. Consider circles centered at $p$ which separate $p$ and $q$.  Then there is some such circle $C(p,r)$ that separates infinitely many $p_i$ from the associated $q_i$, and which separates $p$ from $q$.  Now consider arcs $\alpha_i$ in $\gamma_i(B)$ which connect $p_i$ and $q_i$.  Let $z_i$ be a point on $\alpha_i \cap C(p,r)$.  Then the distances $d(z_i, \partial\gamma_i(B))$ are going to zero since the $\partial \gamma_i(B) \cap C(p,r)$ are accumulating in $C(p,r)$. However, the distances from $z_i$ to $p_i$ and from $z_i$ to $q_i$ are bounded strictly above zero.  This is because we may assume that the the $p_i$ are contained in a closed disk, which has positive distance from $C(p,r)$.  We may assume the same thing for the $q_i$.  Since the lengths of the arcs of $\alpha_i \setminus z_i$ are bounded below by the distances $d(z_i, p_i)$ and $d(z_i, q_i)$, this contradicts property 2. above of a uniform domain in Definition \ref{d:uniform}.  Since quasi-discs are uniform domains \cite[Part I, Thm 6.2]{Lehto}, the image of a boundary curve cannot accumulate.

(3) Next we claim that the image of a boundary curve $c$ of $\Niel_{C_1}(C_1,..., C_n)$ does not exit a cusp.

Suppose that the image of a boundary curve $\beta$ of $\Niel_{C_1}(C_1,...,C_n)$ does exit a cusp.  Then in $C_1$, one of the endpoints of $\beta$ in $\partial C_1$ is the fixed point of a parabolic element $\gamma_p \in \Stab(C_1)$. Call this point $p$.  Let $B$ be another component of $\Omega(\Gamma)$ whose boundary contains $p$.  

We first claim that $\gamma_p$ also stabilizes $B$.  Indeed, suppose not and conjugate so that $\gamma_p$  fixes the point at infinity and translates by the action $z \rightarrow z+1$. If $\gamma_p$ does not stabilize $B$, the interiors of $\gamma_p^n(B)$ are distinct for $n \in \mathbb{Z}$.  The quasicircles $\partial\gamma_i(B)$ all go through infinity.  Therefore, since the transformations $\gamma_p$ and $\gamma_p^{-1}$ are translations that take $B$ off of itself, any point of $B$ is at most distance 1 from $\partial B$. Since $\partial B$ goes through $\infty$, there are points $z_1$ and $z_2$ of $B$ such that $|z_1 -z_2| > 2b$, for any constant $b$.  Then any arc $\alpha$ connecting $z_1$ and $z_2$ contains a point $z$ (the midpoint) such that 
$min(l(\alpha_1), l(\alpha_2)) \geq 1/2|z_1-z_2| > b \geq bd(z, \partial B)$,
where $\alpha_1$ and $\alpha_2$ are the components of $\alpha \setminus z$. This is a contradiction, since $B$ is a quasidisc and hence a uniform domain. 

Thus $\gamma_p$ stabilizes $B$, where $B$ is any component of $\Omega(G)$ such that $p \in \partial B$.   But this contradicts the assumption that $\beta$ is a boundary curve. Indeed, let $q$ be the other endpoint of $\beta$. Then $\gamma_p^n(q)$ and $\gamma_p^{-n}(q)$ will approach $p$ from both sides. Since $\gamma_p$ stabilizes $\Bump(C_1,...C_n)$,  $\beta$ cannot be a boundary curve of the convex hull of this set.   Therefore, the image of a boundary curve $\beta$ of $\Niel_{C_1}(C_1,...,C_n)$ does not exit a cusp.  

Since the image of a boundary curve is simple, does not accumulate, and does not exit a cusp, it is a simple closed curve. It is a geodesic in the Poincar\'e metric on $C_1$ as it is the boundary of a convex hull of points on the boundary $\partial C_1$.  Since $p: {\bf H}^3 \cup \Lambda(G) \rightarrow M(G)$ is a covering map, the image of boundary curves are boundary curves of the image.  This proves Theorem \ref{thm:boundary}.  

\end{proof}
\section{The characteristic submanifold} 

The characteristic submanifold of $(M(G), \partial M(G))$ is a 3-submanifold $(X_M,S_M)$ of $(M(G), \partial M(G))$ such that all of the essential tori and annuli in $M(G)$ can be property isotoped into $(X_M, S_M)$. It was defined and studied extensively by Jaco and Shalen \cite{JS} and Johannson \cite{Joh}. See also \cite[1.8]{Misha}.  It is defined by the following properties. 

(1) Each component $(X,S)$ of $(X_M, S_M)$ is an $I$-bundle over a surface or a solid torus equipped with a Seifert fibered structure. 

(2) The components of $\partial X \setminus \partial M(G)$ are essential annuli.

(3) Any essential annulus (or M\"obius band) is properly homotopic into $(X_M,S_M)$. 

(4)  $(X,S)$ is unique up to isotopy. 

A component of the characteristic submanifold which is a Seifert fibered solid torus can be described from the point of view of the domain of discontinuity as in the following example.  Suppose that there are three components $A$ $B$ and $C$ of the domain of discontinuity $\Omega(G)$ such that $\bar A \cap \bar B \cap \bar C = \lbrace p,q \rbrace $ and $\phi$ is a pseudo-Anosov element of $G$ which fixes $p$ and $q$ such that  $\phi(A) =B$, $\phi(B) =C$ and $\phi(C) =A$.  Thus $\phi^3 \in \Stab(A)$.  Now consider the solid torus $T$ which is the quotient of a regular neighborhood of the geodesic in $ {\bf H}^3$ with endpoints $p$ and $q$.  There is also an annulus in $p(A)$ which is the quotient of a regular neighborhood of the image of the geodesic in $A$ with endpoints $p$ and $q$.  This is stabilized by $\phi^3$. Suppose further that $A$ $B$ and $C$ are the only components of $\Omega(G)$ whose closures meet $p$ and $q$.  This annulus in $p(A)$ and an annulus on $\partial T$ co-bound an annulus $\times I$.  The annulus on $\partial T$  wraps three times around in the direction invariant by $\phi$. Then the component $X$  which is $T$ union the annulus $\times I$ is  a component of the characteristic submanifold.  This is a solid torus fibered by circles which wind three times around the core.  The boundary of $X$ is an annulus on $\partial M(G)$ union an annulus in the interior of $M(G)$.  Thus $\partial X \setminus \partial M(G)$ is an annulus.  This is one of the cases described in the proof of Theorem \ref{characteristic} below.  For now we give another example, where the union of the convex hulls of the bumping sets  is all of $\Omega(G)$.    Figure \ref{ptorus} shows the limit set of a quasi-Fuchsian free group of rank 2 acting on $\mathbb{C} \cup \infty = S^2_\infty$.  This picture was made with Curt McMullen's lim program \cite{lim}. The stabilizer of either component is the whole group.  Figure \ref{adjoin} shows what happens when we adjoin the square root of one of the generators.  If we denote the square root by $\gamma$, then $\gamma$ switches two components of $\Omega(G)$ which meet at the endpoints of the geodesic invariant by $\gamma$.  These are the center and outer components in Figure \ref{adjoin}.  $\gamma^2$ is in the stabilizer of both.  As above, there is a component of the characteristic submanifold which is a Seifert fibered torus and whose boundary is a union of two annuli.  In this case, we can also think of this component as a twisted $I$-bundle.

\begin{figure}
\begin{minipage}[b]{0.5\linewidth}
\centering
\includegraphics[width=\textwidth]{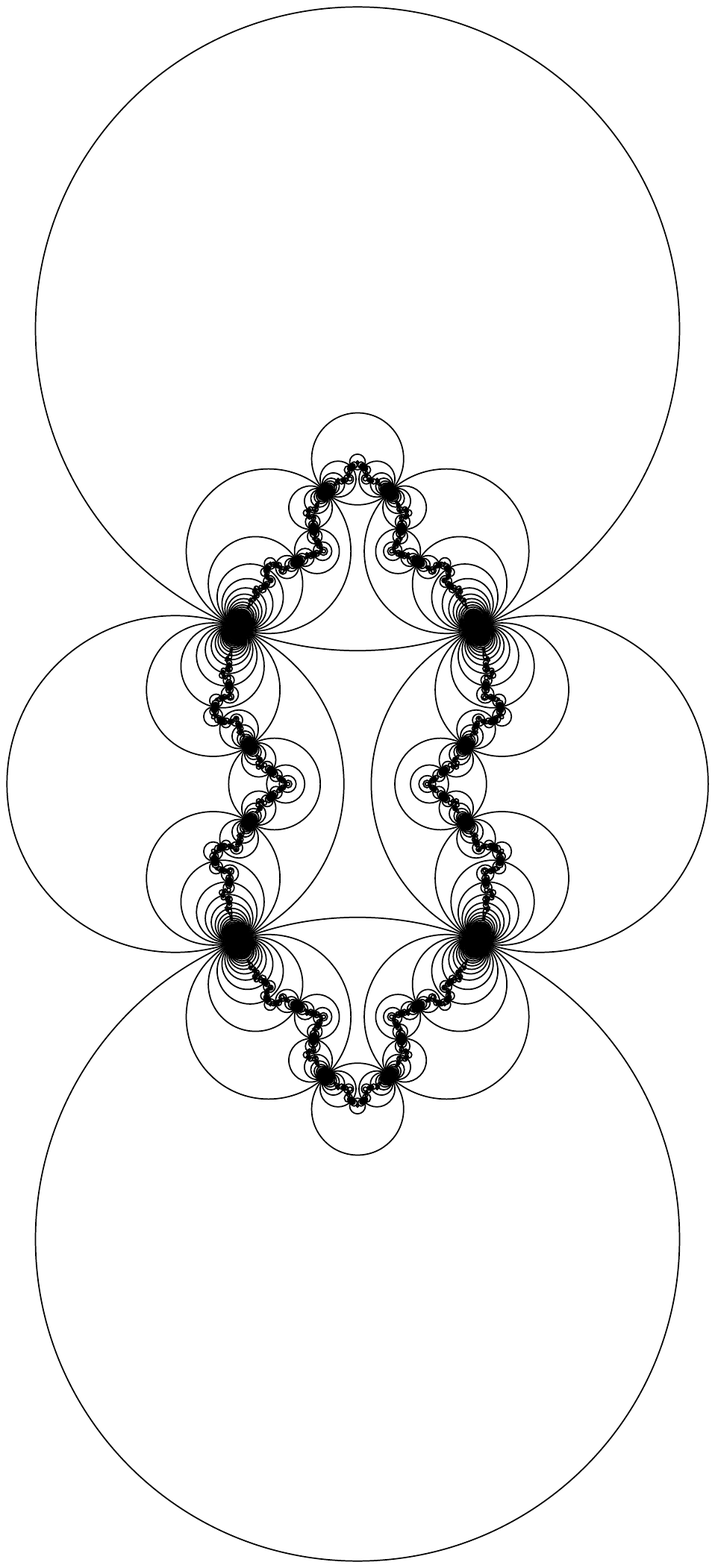}
\caption{A quasi-Fuchsian punctured torus.}
\label{ptorus}
\end{minipage} \hfill
\hspace{.5cm}
\begin{minipage}[b]{0.45\linewidth}
\centering
\includegraphics[width=\textwidth]{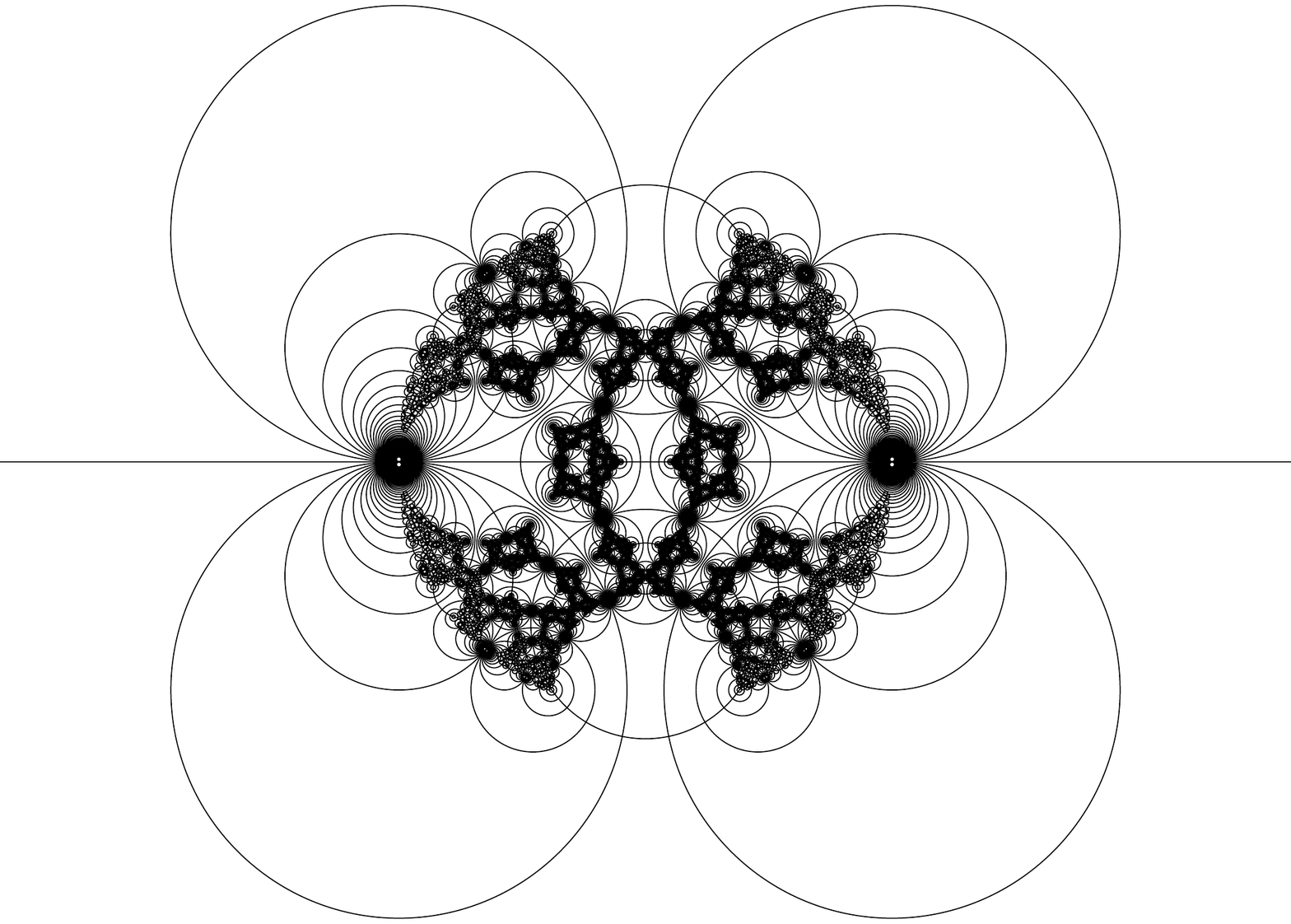}
\caption{After adjoining a square root.}
\label{adjoin}
\end{minipage}
\end{figure}

We will show here that the Nielson core of the bumping set, taken over all components of $\Omega(G)$, is the boundary of the characteristic submanifold of $M(G)$.  More precisely,

\begin{theorem} \label{characteristic} Let $G$ be a geometrically-finite Kleinian group with incompressible boundary and let $p: {\bf H}^3 \cup \Omega(G) \rightarrow M(G)$ be the covering map induced by the action of $G$. Then consider

$$ S'= \bigsqcup_C p(Niel_C(\mathcal{C})) $$ 

where $C$ ranges over the components of $\Omega(G)$ and $\mathcal{C }$ is a collection of components of $\Omega(G)$ containing $C$ which has non-trivial bumping set.  We require that each collection be maximal in the sense that adding any other components would strictly decrease the bumping set.  Let $S$ be $S'$ with any simple closed curves replaced by regular annular neighborhoods of these curves.  Then $S$, considered as a disjoint union of components as above, is the boundary of the characteristic submanifold of $M(G)$.  

 \end{theorem}
 
 To this end, we will need the following lemma: 
 
 \begin{lemma} \label{lem:stab} Let $G$ be a geometrically finite Kleinian group with incompressible boundary.  Let $C$ be a component of $\Omega(G)$. Assume $\phi \in \Stab(C)$ and that the fixed points of $\phi$ in $\partial C$ are in $\Bump(C,B)$.  Then $\phi \in \Stab(B)$. 
 \end{lemma} 
\begin{proof} 
The transformation $\phi$ is either parabolic or hyperbolic, by which we mean either strictly hyperbolic or loxodromic.  The parabolic case is contained in (3) of the proof of Theorem \ref{thm:boundary}.  Suppose that $\phi$ is hyperbolic. We may assume that $\phi$ fixes $0$ and $\infty$. Then the boundary of $B$, and the boundary of $\phi^n(B)$ for all $n$, meet $0$ and $\infty$.  If the $\phi^n(B)$ are all distinct, then they accumulate along the circle $|z|=1$.  Thus there are pairs of points in the $\phi_n(B)$, with one element of the pair close to 0 and the other on the circle $|z|=2$ which cannot be connected by any arc satisfying (2) of Definition \ref{d:uniform}. This is because any arc connecting such a pair would have to pass through the circle $|z|=1$, and these points are arbitrarily close to the boundary. Thus $\phi^n(B) =B$ for some $n$. 

Now suppose that $\phi(C) =C$. This does not preclude $\phi$ from being loxodromic.  However, $\phi$ leaves $S^2_\infty \setminus \bar C$ invariant, and this domain is conformally equivalent to the hyperbolic plane. Then $\phi$ is conjugate to $z \rightarrow \lambda z$, where $\lambda$ is real. By abuse of notation, we continue to denote the transformation $z \rightarrow \lambda z$ by $\phi$ and we denote the image of $B$ by $B$.  Note that $\phi$ leaves rays from the origin invariant.  Since $\phi^n(B) = B$, the arcs of $B \cup r$ will be linked with the arcs of $\phi(B) \cup r$ along some ray $r$ from the origin.  This is a contradiction as $B$ and $\phi(B)$ are disjoint and connected. 
\end{proof}

We now state the annulus theorem in this setting.  See \cite{CF} for the general case.   The proper immersed image $\mathcal{A}$ of an annulus or M\"obius band in a hyperbolic manifold $M$ with boundary is \emph{essential} if it induces an injection on the level of fundamental groups, and if it is not properly homotopic into a cusp neighborhood. 

\begin{theorem} \cite{CF} \label{thm:annulus} Let $G$ be a Kleinian group with incompressible boundary and let $p: {\bf H}^3 \cup \Omega(G) \rightarrow M(G)$ be the covering map. Let $A$ be an proper immersed essential annulus or M\"obius band in $M(G)$ with embedded boundary.  Then there is a proper embedded essential annulus or M\"obius band $\mathcal{A}$ with the same boundary and a pre-image $\mathcal{\tilde A}$ in ${\bf H}^3 \cup \Omega(G)$ with boundary in two different components of $\Omega(G)$. 

\end{theorem}
That any pre-image has boundary in two different components of $\Omega(G)$ follows immediately from the fact that $\mathcal{A}$ is essential.  We add this to the statement only because it is important for our point of view. 

We now give the proof of Theorem \ref{characteristic}.  Given a Kleinian group $G$ with incompressible boundary, we will form a submanifold $N$ of $M(G)$ which is a characteristic submanifold.  We form this submanifold in pieces, considering maximal collections of components of $\Omega(G)$ which meet in a given bumping set.  Note that if there are more than two components in such a collection, the bumping set must be exactly two points, as every circle on $S^2_\infty$ is separating.  (We ignore collections that bump in exactly one point, since the convex hull of the bumping set will be trivial in this case.) In constructing these pieces, we will show that they satisfy properties (1) and (2) of the definition of characteristic submanifold above.  Then we will show (3) that any essential annulus or M\"obius band is properly homotopic into one of these components. That the result is unique up to isotopy follows from the fact that it is a characteristic submanifold. 

Note that we are considering the disjoint union of the components in the statement of Theorem \ref{characteristic}.  To have the union in $M(G)$ consist of disjoint components, the components may need to be pushed slightly off of each other.

{\bf Components of the characteristic submanifold obtained by the bumping of two components:} We first consider two components $C$ and $D$ which bump, and which bump in exactly two points $p$ and $q$. We assume maximality in that there are no other components of the domain of discontinuity which meet both $p$ and $q$. In this case $\Niel_C(C,D)$ consists of one arc in $C$, $\tilde l$, which is invariant under some element $g \in G$ by Thm \ref{thm:boundary}.  The element $g$ fixes both $p$ and $q$ on $S^2_\infty$.  We choose $g$ so that it is primitive in the stabilizer of $C$.  By Lemma \ref{lem:stab} $D$ is also stabilized by $g$, $g$ is primitive in $\Stab(D)$, and the arc $\tilde l' = \Niel_D(C,D)$ is also invariant under $g$.  Then $l = p(\tilde l)$ and $l' = p(\tilde l') $ are freely homotopic through the manifold $M(G)$ to the closed geodesic $l_{int}$ in $M(G)$  which lifts to a geodesic $\tilde l_{int}$ in $\mathbb{H}^3$ with endpoints $p$ and $q$ which is invariant under $g$.  If $l =l'$, then this homotopy will define an immersed M\"obius strip.  (In this case there is an $f \in G$ such that $f^2=g$.) If $l \neq l'$, this homotopy defines an immersed annulus.  In either case, by the annulus theorem, there is an embedded essential M\"obius strip or annulus $A$ with the same boundary.  When $A$ is an annulus, taking a regular neighborhood of this annulus gives us  an  $(A \times I, S^1 \times S^0 \times I)$ as a component $(X,S)$ of the characteristic submanifold  $(X_M,S_M)$. When $A$ is a M\"obius strip, we get a component $(X,S)$ which is a twisted $I$-bundle over an annulus. We can also realize this case as $(T, S)$, where $T$ is a solid torus with a natural Seifert fibered structure and $S$ is an annulus.  There is a such a solid torus component of the characteristic submanifold of the 3-manifold illustrated in Figure \ref{adjoin}.   

We now consider two components $C$ and $D$ which bump, and whose bumping set contains more than two, and hence infinitely many, points.   When $\Bump(C,D)$ contains more than 2 points, $\Niel_C(C,D)$ contains more than just a single geodesic.  If $\Bump(C,D) = \partial C = \partial D$, then the characteristic submanifold is the entire manifold $M(G)$ which is a $I$-bundle over a surface.  Otherwise the image $p(\Niel_C(C,D))$ of the convex hull of the bumping set is a subsurface of $p(C)$ bounded by geodesics, by Theorem \ref{thm:boundary}.  We claim that for each boundary curve of $l$ of $p(\Niel_C(C,D))$, there is a boundary curve $l'$ of $p(\Niel_D(C,D))$ and an essential annulus $A$ with $\partial A = l \cup l'$. Indeed, there is a lift $\tilde l$ of $l$ that is a boundary curve of $\Niel_C(C,D)$.   By Theorem \ref{thm:boundary} and Lemma \ref{lem:stab}, $\tilde l$  is stabilized by some element $g$ of $G$ which stabilizes $C$. Let $p$ and $q$ be the fixed points of $g$ on $S^2_\infty$.  By Lemma \ref{lem:stab}, the boundary geodesic $\tilde l'$ of $\Niel_D(C,D)$ in $D$ with endpoints $p$ and $q$ is also stabilized by $g$.  Let $p(\tilde l') = l'$.  Then $l $ and $l'$ are freely homotopic, since they are both homotopic to the geodesic representing $g$.  If $l \neq l'$, then by the Annulus Theorem, there is an embedded annulus $A$ with boundary $l$ and $l'$.   If $l = l'$, then there is an $f \in G$ such that $f(\tilde l) = \tilde l'$, where $f$ has the same fixed points as $g$.  Since $l$ is a boundary curve of $\Niel_C(C,D)$, $f(D) \neq C$, so $p$ and $q$ are contained in a bumping set involving at least $C$, $D$, and $f(D)$.  In this case, we replace the boundary arcs $\tilde l$ and $\tilde l'$ with $g$-equivariant arcs also called $\tilde l$ and $\tilde l'$  that lie just in the interior of $\Niel_C(C,D)$  and $\Niel_D(C,D)$. We replace $\Niel_C(C,D)$ and $\Niel_D(C,D)$ with the new, shrunken regions.  Then $l \neq l'$ and we can form our embedded annulus with these new curves. We do this for each boundary curve of $p(\Niel_C(C,D))$. Note that some boundary curves will correspond to the same annulus if $p(\Niel_C(C,D)) = p(\Niel_D(C,D))$.  Consider the resulting union of annuli. We claim that we may assume the union is embedded. Firstly, the boundaries of the family of annuli do not intersect by construction. Secondly, we can remove any inessential circles of intersection by an innermost disk argument as $M(G)$ is irreducible.   Thirdly, there can be no essential intersections.  Indeed,  any such  curve of intersection must lift to an arc in $\mathbb{H}^3$ which meets the limit set of $G$ in the same two points as the boundary components of two different annuli.  But the lifts of the boundaries of the allegedly essentially  intersecting  annuli meet the limit set in different points of $\Bump(C,D)$.  This is because they correspond to different pairs of boundary curves of $p(\Niel_C(C,D))$ and $ p(\Niel_D(C,D))$.  Therefore, we may assume that the collection of annuli connecting the boundary components of $p(\Niel_C(C,D))$ and $p(\Niel_D(C,D))$ is embedded.  This family of annuli lifts to an embedded family of strips $\mathcal{R}^2 \times I$ in $\mathbb{H}^3 \cup \Omega(G)$.  There will be some region $R$ bounded by these strips which meets $\Niel_C(C,D)$.  The image $p(R)$ is a component $(X,S)$ of $(X_M, S_M)$.  It is an $I$-bundle over a surface.  When $p(\Niel_C(C,D)) = p(\Niel_D(C,D))$, this will be a twisted $I$-bundle. The components of $\partial X \setminus \partial M$ are the annuli constructed above.

{\bf Components of the characteristic submanifold obtained by the bumping of more than two components:} We now consider the case when there are more than two components of the domain of discontinuity which bump non-trivially.  In this case the closures of the components must meet in exactly two points.  Again we assume that adding any components of $\Omega(G)$ to the collection results in a smaller bumping set. Let $C_1,...C_n$ be this maximal collection whose closures meet in two points $p$ and $q$.  By Theorem \ref{thm:boundary} and Lemma \ref{lem:stab}, there is a $g \in G$ such that for each $C_i$, there is a geodesic $\tilde l_i$ which is stabilized by $g$.  Denote the geodesic in $\mathbb{H}^3$ stabilized by $g$ with endpoints $p$ and $q$ by $\tilde l_{int}$, with image $l_{int}$. Either (1) all the $p(\tilde l_i) = l_i$ are the same, (2) all the $p(\tilde l_i) = l_i$ are different, or (3) the images fall into $m$ classes, where $m | n$.  We deal with each situation in turn. 

(1) If all the $l_i = l_1$ are the same, then consider a regular neighborhood $N(l_1)$ of $l_1$.  This lifts to regular neighborhoods of each of the $\tilde l_i$.  The boundary curves of these regular neighborhoods end in $p$ and $q$.  Orient $N(l_1)$ so that there is a left side $\partial N(l_1) _-$ and a right side $\partial N(l_1)_+$.  Since each of $\partial N(l_1) _+$ and $\partial N(l_1) _-$ is freely homotopic to $l_{int}$, the lifts of $\partial N(l_1) _+$ and $\partial N(l_1) _-$  in each $C_i$ bound strips $\mathbb{R} \times [0,1]$ with $\tilde l_{int}$. Then consider two strips in $\mathbb{H}^3 \cup \Omega(G)$, one which is bounded by a lift of $\partial N(l_1)_+$ and $\tilde l_{int}$, and the other which is bounded by a lift of $\partial N(l_1)_-$ and $\tilde l_{int}$, where the lifts of $\partial N(l_1)_+$ and $\partial N(l_1)_-$ are in different components of $\Omega(G)$.  Then the union of these two strips will map down to an essential annulus in $M(G)$, and by the annulus theorem there is an embedded essential annulus $\mathcal{A}$ with the same boundary, which is $\partial N(l_1)_+ \cup \partial N(l_1)_-$. The pre-image of $\mathcal{A}$ is an embedded collection of strips each of which meets $\Omega(G)$ in two different components.  Order the components $C_i$ cyclically around $\tilde l_{int}$.   Since $\mathcal{A}$ is embedded, the strips must connect $C_i$ to $C_{i+1} \mod n$.  Then the pre-image of $\mathcal{A}$ will partition $\mathbb{H}^3 \cup \Omega(G)$ into regions, one of which, $R$, will meet $\tilde l_i$.  The region $R$ is a regular neighborhood of $\tilde l_{int}$ union thickened strips which meet the lifts of $N(l_1)$.  It is a naturally fibered by $g$-invariant lines. (Note that $\tilde l_{int}$ is $f$-invariant, where $f^n =g$.)  The image of $R$ is a component of the characteristic submanifold, $(X, S) = (p(R), N(l_1))$.  It is Seifert-fibered by the images of the $g$-invariant lines and $\partial X \setminus \partial M(G)$ is the annulus $\mathcal{A}$.

(2) We consider the case when the $l_i$ are all distinct. Denote regular neighborhoods of these curves by $N(l_i)$ and lifts  which meet $p$ and $q$ by $N(\tilde l_i)$. As above, order the $C_i$ around $\widetilde l_{int}$ and label the two boundary components of $N(\tilde l_i)$ by $\partial N(\tilde l_i)_+$ and $\partial N(\tilde l_i)_-$ so that $\partial N(\tilde l_i)_+$ is next to  $\partial N(\tilde l_{i+1})_-$ in this cyclic ordering.  Note that the $C_i$ may be swirling around $p$ and $q$ as they approach them (if $g$ is loxodromic) but we can choose some circle on $S^2_\infty$ separating $p$ and $q$ and cyclically order the components with respect to this circle, and this is well-defined up to the orientation of $S^2_\infty$.  Then as above there are $g$-invariant strips in $\mathbb{H}^3 \cup \Omega(G)$ connecting each component of $\partial N(\tilde l_i)$ to $\tilde l_{int}$. The union of two such $g$-invariant strips, one from $\partial N(\tilde l_i)_+$ to $\tilde l_{int}$ and one from $\tilde l_{int}$ to $\partial N(\tilde l_{i+1})_-$ map down to an immersed essential annulus in $M(G)$.  By the annulus theorem, there is an embedded essential annulus with the same boundary, $\mathcal{A}_i$.  This has a lift $\widetilde{\mathcal{A}_i}$ which meets $p$ and $q$.  Since all the $l_i$ are distinct, this $\widetilde{\mathcal{A}_i}$ has boundary $\partial N(\tilde l_i)_+$ and $\partial N(\tilde l_{i+1})_-$.  We form such an embedded annulus $\mathcal{A}_i$ for each $i \mod n$, with a lift $\widetilde{\mathcal A}_i$ with boundary  $\partial N(\tilde l_i)_+$ and $\partial N(\tilde l_{i+1})_-$  which approaches $p$ and $q$.  We claim that we can choose such annuli so that the union is embedded.  Firstly, the union of the boundaries is already embedded.  Secondly, remove any circles of intersection which are trivial in some (hence any) annulus by incompressibility and irreducibility.  Now consider any remaining circles of intersection between the $\mathcal{A}_i$ and $\mathcal{A}_1$ in $\mathcal{A}_1$.  These are parallel, essential curves on $\mathcal{A}_1$.  Hence in the lift $\widetilde{\mathcal{A}_1}$, the lifts of these intersections all approach $p$ and $q$.  This means that the lift $\widetilde{\mathcal{A}_1}$ only intersects the $\widetilde{\mathcal{A}_i}$ which approach $p$ and $q$.  As these are not linked in the cyclic ordering around $p$ and $q$, there is some pair of intersection curves which bounds annuli on both $\mathcal{A}_1$ and some $\mathcal{A}_i$.  Switching these two inner annuli and pushing off will reduce the number of intersection curves.   Hence by choosing the collection $\mathcal{A}_i$ to minimize the number of intersection curves, the collection will be embedded.  All the pre-images of the embedded collection $\mathcal{A}_i$ will partition $\mathbb{H}^3 \cup \Omega(G)$ into regions which do not overlap in their interiors.  One of these regions, $R$, will meet the $N(\tilde l_i)$.  This region is naturally foliated by $g$-invariant lines.  The image $(p(R), \bigcup N(l_i))$ is a component of the characteristic submanifold which is Seifert-fibered by the images of the $g$-invariant lines.  $p(R)$ is a solid torus and $\partial(p(R)) \setminus \bigcup N(l_i)$ is the union of the annuli $\mathcal{A}_i$. 

(3) Lastly we consider the case when the images of the $\tilde l_i$ are $m$ distinct curves, where $m | n $ and $ m \neq 1,n$.  Our first task is to show that in this case $l_{int}$, as defined above, is embedded.  As before, each $\tilde l_i$ is invariant under $g \in G$, where $g$ is hyperbolic and fixes $p$ and $q$ on $S^2_\infty$. Denote the geodesic in $\mathbb{H}^3$ invariant under $g$ by $\tilde l_{int}$.  Then, as there are $w = n/m $ curves  in $\tilde l_i$ which are identified in the quotient, $\tilde l_{int}$ is invariant under $f$ where $f^w =g$.  Now let $\tilde l_1, \tilde l_2, ... \tilde l_w$ be the $f$-orbit of $\tilde l_1$, cyclically ordered around $p$ and $q$ as above.  Let $l_1 = p(\tilde l_1)$ and let $N(l_1)$ be a regular neighborhood of this image.  Then label the two boundary components of $N(l_1)$ by $\partial N(l_1)_+$ and $\partial N(l_1)_-$ so that the induced labeling of the boundary components of the lifts has   $\partial {N(\tilde l_i)_+}$ next to $\partial {N(\tilde l_{i+1})_-}$ $\mod w$ in the cyclic ordering around $p$ and $q$.  Then there is a $g$-invariant strip connecting $\partial {N(\tilde l_1)_+}$ and $\tilde l_{int}$ and another connecting $\tilde l_{int}$ with $\partial {N(\tilde l_2)_-}$.  The union of these two invariant strips maps down to an immersed essential annulus in $M(G)$ and by the annulus theorem, there is an embedded essential annulus $\mathcal{A}_{temp}$ with the same boundary.  The lifts of $\mathcal{A}_{temp}$ do not intersect and hence the lifts meeting $p$ and $q$ consist of $w$ strips connecting each $\partial {N(\tilde l_i)_+}$ to $\partial {N(\tilde l_{i+1})_-}$.  The action of $f$ permutes these $w$ strips cyclically and takes $\tilde l_{int}$ to itself. Therefore, $\tilde l_{int}$ is on the inside of these strips.  That is, there is a region $R$ bounded by preimages of $\mathcal{A}_{temp}$ which meets the ${N(\tilde l_i)}$ and $R$ contains $\tilde l_{int}$.  Thus $\tilde l_{int}$ intersects its images under $G$ in either itself or the empty set, which implies $p(\tilde l_{int}) = l_{int}$ is embedded.  Now consider the whole set of $\tilde l_i$. Order them cyclically around $p$ and $q$.  Each $\tilde l_i$ is connected to $\tilde l_{int}$ by a $g$-invariant strip. Since $l_{int}$ is embedded, it has a regular neighborhood $N(l_{int}$ which is a solid torus.  The image of such a $g$-invariant strip in $M(G)$ will restrict to an essential proper immersed annulus in $M(G) \setminus N(l_{int})$.  By the annulus theorem, there is an embedded annulus $\mathcal{A}_{i, int}$ in $M(G) \setminus N(l_{int})$ with the same boundary. We form such an annulus for each $l_i$.   Note that each boundary on $\partial N(l_{int})$ is a curve of the same slope.  Therefore, we can arrange so that the boundaries of the $\mathcal{A}_{i,int}$ are disjoint, and cyclically ordered in the same order as the $\tilde l_i$.  Then since the boundary of the annuli are not linked, by choosing a collection that intersect minimally, the $\mathcal{A}_{i,int}$ will be disjoint.  Now take a regular neighborhood of $N(l_{int}) \cup_i \mathcal{A}_{i,int}$.  This will be a solid torus $W$ which meets the boundary of $M(G)$ in $n$ parallel annuli $N(l_i)$, where $N(l_i)$ is a regular neighborhood of $l_i$.  The component of the characteristic submanifold will be $(W, \cup N(l_i) )$.  This has a pre-image $R$ in $\mathbb{H}^3 \cup \Omega(G)$ which meets the $ N(\tilde l_i)$ and which is naturally foliated by $g$-invariant lines.  The images of these lines in $W$ are a Seifert-fibering of the solid torus $W$.  The components of $\partial W \setminus \cup N(l_i)$ are annuli on $\partial W$. The pre-images in $R$, as before, connect neighboring boundary components of the $ N(\tilde l_i)$.  

Now let $\mathcal{B}$ be an essential annulus or M\"obius strip in $M(G)$.  We will show that $\mathcal{B}$ is properly homotopic into the submanifold constructed above.  Pick a basepoint on $\mathcal{B}$ and let $g$ generate the fundamental group of $\mathcal{B}$ in $G$.  Since $\mathcal{B}$ is essential, a lift $\widetilde{\mathcal{B}}$ of $\mathcal{B}$ must meet two different components, $C$ and $D$, of $\Omega(G)$, both of which are $g$-invariant and which meet the fixed points $p$ and $q$ of $g$ on $S^2_\infty$. Thus $C$ and $D$ bump at $p$ and $q$.  From our construction, there is a $g$-invariant strip $\widetilde{A}$ contained in some component $(X,S)$ which meets $C$ and $D$ in the convex hull of $p$ and $q$ in each component.  Consider the solid torus $T = (\mathbb{H}^3 \cup S^2_{\infty} \setminus \lbrace p,q \rbrace)/<g>$.  This is a solid torus and the images of $\widetilde{B}$ and $\widetilde{A}$ are two embedded essential annuli with the same slope. They are therefore parallel by a proper isotopy. This isotopy maps down to a proper homotopy of $\mathcal{B}$ into a component $(X,S)$ of the submanifold we have constructed.

\qed

\subsection{Acknowledgements} 

I would like to thank Curt McMullen and Ian Biringer for encouragement.  I am also grateful to the NSF for partial support via grant 0805908. 

\end{document}